\documentclass{article}
\usepackage[a4paper,textwidth=15cm,textheight=21.5cm,centering]{geometry}
\usepackage[T1]{fontenc}
\usepackage[utf8]{inputenc} 
\usepackage{lmodern}
\usepackage{mathtools}   
\usepackage{amssymb, amsfonts}
\usepackage{amsthm}
\usepackage{bm}
\usepackage{graphicx}
\usepackage{booktabs, dcolumn}
\usepackage[tableposition=top]{caption}
\usepackage{tikz}
\usetikzlibrary{arrows.meta,calc,positioning}
\usepackage{enumerate}
\usepackage[shortlabels]{enumitem}
\usepackage{xcolor}
\usepackage{hyperref}
\hypersetup{
    colorlinks=true,
    citecolor=blue,
    linkcolor=blue,
    filecolor=magenta,      
    urlcolor=cyan,
}
\usepackage{bookmark}   
\numberwithin{equation}{section}
\theoremstyle{plain}
\newtheorem{theorem}{Theorem}[section]
\newtheorem{proposition}[theorem]{Proposition}
\newtheorem{lemma}[theorem]{Lemma}
\newtheorem{corollary}[theorem]{Corollary}

\newtheorem{claim}[theorem]{Claim}

\theoremstyle{definition}
\newtheorem{definition}[theorem]{Definition}

\theoremstyle{remark}

\newcommand{\Om}{\Omega}

\newcommand{\R}{\mathbb{R}}

\newcommand{\G}{\Gamma}

\newcommand{\sph}{\mathbb{S}}
\newcommand{\sub}{\subset}

\newcommand{\Ric}{\operatorname{Ric}}
\newcommand{\Bry}{\operatorname{Bry}}
\begin{document}
\title{$\kappa$-solutions with the round cylinder as an asymptotic shrinker}
\author{Aprameya Girish Hebbar\thanks{Department of Mathematics, Rutgers University, Piscataway, NJ 08854. Email: \texttt{ah1531@math.rutgers.edu}.}}

\date{}
\maketitle
\begin{abstract}
We show that $\kappa$-solutions to the Ricci flow in dimensions $n\ge 4$ whose asymptotic shrinking Ricci soliton is the round cylinder $\mathbb S^{n-1}\times\mathbb R$ must be uniformly PIC. Combined with earlier classification results, this implies that any such noncompact solution is either the round shrinking cylinder or the Bryant steady soliton, and any such compact solution is Perelman's ancient solution. 
\end{abstract}
\section{Introduction}
The Ricci flow, introduced by Hamilton \cite{Ham82}, evolves a Riemannian metric by
$$
\partial_t g(t)=-2\,\Ric_{g(t)}.
$$
It has become one of the central tools in geometric analysis and topology, most notably through Perelman's resolution of the Poincar\'e and Geometrization conjectures \cite{Per02,Per03}. A fundamental problem in the theory is to understand the formation of singularities. Such singularities are analyzed by rescaling the flow around regions of high curvature, and the possible blow-up limits are ancient solutions which are noncollapsed. Among these, Perelman's $\kappa$-solutions are of particular importance.
\begin{definition}
An ancient solution $(M^n,g(t))_{t\leq 0}$ of the Ricci flow is called a $\kappa$-solution if it is complete and nonflat, has bounded nonnegative curvature operator, and is $\kappa$-noncollapsed at all scales in the sense of Perelman.
\end{definition}
In dimension $3$, the $\kappa$-solutions have been completely classified. Brendle proved that any noncompact three-dimensional $\kappa$-solution is isometric to either a family of round shrinking cylinders, a quotient thereof, or the Bryant soliton \cite{Bre20}. In the compact case, Brendle--Daskalopoulos--\v{S}e\v{s}um proved that any compact three-dimensional $\kappa$-solution on $\mathbb S^3$ is either the family of round shrinking spheres or Perelman's ancient solution \cite{BDS21}.

In higher dimensions, the situation is substantially more delicate, and a general classification of $\kappa$-solutions in this sense remains a major open problem. We refer the reader to Haslhofer \cite{Has24} for recent progress in dimension four. Under the additional assumption of uniformly positive isotropic curvature (PIC), however, classification results are available. Indeed, since nonnegative curvature operator implies weakly PIC2, the classification theorems of Brendle--Naff and Li--Zhang apply to the noncompact uniformly PIC $\kappa$-solutions considered here: every such solution is either the round shrinking cylinder, a quotient thereof, or the Bryant soliton \cite{LZ22,BN23}. In the compact case, Brendle--Daskalopoulos--Naff--\v{S}e\v{s}um proved that any uniformly PIC compact $\kappa$-solution diffeomorphic to $\mathbb S^n$ is rotationally symmetric and is either the family of shrinking round spheres or Perelman's ancient solution \cite{BDSN23}.

In this paper we study $n$-dimensional $\kappa$-solutions under the additional assumption that Perelman's asymptotic shrinker is the standard round cylinder $\mathbb S^{n-1}\times \mathbb R$. Li--Zhang \cite{LZ22} studied such solutions and established compactness, cap-neck decomposition, and the classification of the rotationally symmetric examples. The purpose of the present paper is to remove the rotational symmetry assumption and treat both the compact and noncompact cases. Our main result is the following.
\begin{theorem}
\label{thm:main-result}
Let $(M^n,g(t))_{t\leq 0}$, $n\geq 4$, be a $\kappa$-solution whose asymptotic shrinker is the standard round cylinder $\mathbb S^{n-1}\times\mathbb R$. Then there exist constants $\Lambda>0$ and $b<0$ such that the flow $(M^n,g(t))_{t\leq b}$ is uniformly $\Lambda$-PIC.
\end{theorem}
Since the flow already has nonnegative curvature operator, Theorem~\ref{thm:main-result} supplies the remaining uniform PIC condition needed to invoke the classification theorems of \cite{BN23,BDSN23}, after shifting time if necessary. We obtain the following classification.
\begin{theorem}
\label{thm:full-classification}
Let $(M^n,g(t))_{t\leq 0}$, $n\geq 4$, be a $\kappa$-solution whose asymptotic shrinker is the standard round cylinder $\mathbb S^{n-1}\times\mathbb R$. Then $(M^n,g(t))_{t\leq 0}$ is, up to parabolic rescaling and time translation, one of the following solutions:
\begin{enumerate}
    \item the family of round shrinking cylinders $\mathbb S^{n-1}\times\mathbb R$;
    \item the $n$-dimensional Bryant steady soliton;
    \item Perelman's ancient solution.
\end{enumerate}
\end{theorem}
Chan--Ma--Zhang \cite{CMZ25b} proved that, for ancient Ricci flows admitting Perelman's asymptotic solitons, these asymptotic solitons coincide with Bamler's tangent flows at infinity. Thus, in that setting, the above classification yields the conclusion mentioned by Haslhofer for four-dimensional $\kappa$-solutions whose tangent flow at $-\infty$ is $\mathbb S^3\times\mathbb R$ \cite{Has24}. 

\medskip
\noindent\textbf{Organization of the paper.}
In Section~\ref{sec:preliminaries} we collect preliminaries on asymptotic shrinkers and positive isotropic curvature. The proof of Theorem~\ref{thm:main-result} is divided into two cases: the noncompact case is handled in Section~\ref{sec:noncompact-case}, and the compact case in Section~\ref{sec:compact-case}.

\medskip
\noindent \textbf{Acknowledgments.} I thank Nata\v{s}a \v{S}e\v{s}um for her insights and continued support. I also thank Max Hallgren, Keaton Naff, and Junming Xie for inspiring discussions. 

\section{Preliminaries}
\label{sec:preliminaries}
We begin by reviewing the notions of an asymptotic shrinker and curvature pinching. Perelman \cite{Per02} proved the existence of an asymptotic shrinker for $\kappa$-solutions. We refer the reader to \cite[Proposition 39.1]{KL08} for a detailed proof of the following theorem. We write $\mathcal L$ for Perelman's reduced length and $\ell$ for the corresponding reduced distance. 
\begin{theorem}
\label{thm:existence-of-asymptotic-shrinker}
Let $n\geq 3$, $(M^n,g(t))_{t\in (-\infty,1]}$ be a $\kappa$-solution and $p_0\in M$. Let $\tau_i\to \infty$ and let $q_{\tau_i}$ be chosen so that for the backward Ricci flow $h(\tau):=g(-\tau),\tau\geq 0$,
\begin{equation}
\label{eqn:existence-of-asymptotic-shrinker-1}\ell^h_{(p_0,0)}(q_{\tau_i},\tau_i)\leq \frac{n}{2}.
\end{equation}
Then, after passing to a subsequence, 
\begin{equation}
\label{eqn:existence-of-asymptotic-shrinker-2}
\left(M^n,{\left(g_i(t):=\frac{1}{\tau_i}g(\tau_it)\right)}_{t\in [-1,-\frac{1}{2}]},(q_{\tau_i},-1)\right)\to (M^n_\infty,(g_\infty(t))_{t\in [-1,-\frac{1}{2}]},(q_\infty,-1)),
\end{equation}
 in the smooth Cheeger--Gromov sense where $(M^n_\infty,(g_\infty(t))_{t\in [-1,-\frac{1}{2}]})$ is a nonflat gradient shrinking Ricci soliton. 
\end{theorem}
We call $M^n_\infty$ an asymptotic shrinker (based at $(p_0,0)$) of $(M^n,g(t))_{t\leq 1}$. 
\begin{definition}
Let $(M^n,g(t))_{t\in (-\infty,1]}$ be a $\kappa$-solution with $n\geq 4$. We say that the flow $(M^n,g(t))_{t\in(-\infty,1]}$ has $\sph^{n-1}\times\R$ as an asymptotic shrinker if there exists $p_0\in M$ such that the asymptotic shrinker based at $(p_0,0)$ is the standard round cylinder $\sph^{n-1}\times \R$. 
\end{definition}
As noted in \cite{LZ22}, if one asymptotic shrinker is $\sph^{n-1}\times \R$, then every asymptotic shrinker must be $\sph^{n-1}\times \R$. 

We first observe that having $\sph^{n-1}\times \R$ as an asymptotic shrinker is preserved if we pass to the universal cover. 
\begin{lemma}
\label{lem:asymptotic-shrinker-of-universal-cover}
Let $n\geq 4$, and let $(M^n,g(t))_{t\leq 0}$ be a $\kappa$-solution which has $\sph^{n-1}\times \R$ as an asymptotic shrinker. Let $\pi:\widetilde M\to M$ be the universal cover of $M$, and set $\tilde g(t):=\pi^*g(t)$. Then $(\widetilde M,\tilde g(t))_{t\leq 0}$ is also a $\kappa$-solution and has $\sph^{n-1}\times \R$ as an asymptotic shrinker.
\end{lemma}
\begin{proof}
For each fixed $t\leq 0$, the map $\pi:(\widetilde M,\tilde g(t))\to (M,g(t))$ is a local isometry. Hence $(\widetilde M,\tilde g(t))_{t\leq 0}$ is a complete ancient solution of the Ricci flow. Moreover, each time-slice $(\widetilde M,\tilde g(t))$ has bounded nonnegative curvature operator, and is nonflat.

Let $\tilde{x} \in \widetilde{M}, x=\pi(\tilde{x})$, $t\leq 0$, and $r>0$. Suppose that $R_{\tilde{g}(t)} \leq r^{-2}$ on $B_{\tilde{g}(t)}[\tilde{x}; r]$. We first claim that $\pi\left(B_{\tilde{g}(t)}[\tilde{x}; r]\right)=B_{g(t)}[x;r]$. The inclusion $\subseteq$ follows since $\pi$ is distance nonincreasing. Conversely, let $y\in B_{g(t)}[x;r]$. Since $(M,g(t))$ is complete, there is a minimizing $g(t)$-geodesic $\gamma$ from $x$ to $y$. Lift $\gamma$ starting at $\tilde x$, and denote the lifted curve by $\tilde\gamma$. If $\tilde y$ is the endpoint of $\tilde\gamma$, then $\pi(\tilde y)=y$, and since $\pi$ is a local isometry,
$$
d_{\tilde g(t)}(\tilde x,\tilde y)
\leq L_{\tilde g(t)}(\tilde\gamma)
=
L_{g(t)}(\gamma)
=
d_{g(t)}(x,y)
\leq r.
$$
Thus $\tilde y\in B_{\tilde g(t)}[\tilde x;r]$, proving the claim.

Since ${R}_{\tilde{g}(t)}=R_{g(t)} \circ \pi$, we obtain $R_{g(t)}\leq r^{-2}$ on $B_{g(t)}[x;r]$. Hence, $\operatorname{Vol}_{g(t)}(B_{g(t)}[x;r])\geq \kappa r^n$. By $\operatorname{Vol}_{\tilde{g}(t)}(B_{\tilde{g}(t)}[\tilde{x};r])\geq \operatorname{Vol}_{g(t)}(B_{g(t)}[x;r])$, it follows that $\operatorname{Vol}_{\tilde{g}(t)}(B_{\tilde{g}(t)}[\tilde{x};r])\geq \kappa r^n$. Thus $(\widetilde M,\tilde g(t))_{t\leq 0}$ is $\kappa$-noncollapsed, and therefore it is a $\kappa$-solution.

It remains to show that $(\widetilde M,\tilde g(t))$ has $\sph^{n-1}\times\R$ as an asymptotic shrinker. By the hypothesis and Theorem~\ref{thm:existence-of-asymptotic-shrinker}, after passing to a subsequence we may choose a basepoint $p_0\in M$, a sequence $\tau_i\to\infty$, and points $q_i:=q_{\tau_i}\in M$ satisfying \eqref{eqn:existence-of-asymptotic-shrinker-1}, such that the convergence in \eqref{eqn:existence-of-asymptotic-shrinker-2} holds with the limit identified with the shrinking round cylinder $\sph^{n-1}\times\R$.

Write $\tilde{h}(\cdot):=\tilde{g}(-\cdot)$. Fix a lift $\tilde p_0\in \widetilde M$ of $p_0$. For each $i$, choose an $\mathcal L$-minimizing curve $\gamma_i$ for $h(\tau)$ from $(p_0,0)$ to $(q_i,\tau_i)$. Let $\tilde\gamma_i$ be its lift starting at $\tilde p_0$, and denote its endpoint by $\tilde q_i$. Then $\pi(\tilde q_i)=q_i$. Since $\pi$ is a local isometry for the backward flows, scalar curvature and speed are preserved along the lifted curve. Hence $\mathcal L_{\tilde h}(\tilde\gamma_i)=\mathcal L_h(\gamma_i).$ Therefore
\begin{equation}
\label{eqn:proof-of-asymptotic-shrinker-of-univ-cover-1}
\ell^{\tilde h}_{(\tilde p_0,0)}(\tilde q_i,\tau_i)
\leq
\frac{1}{2\sqrt{\tau_i}}\mathcal L_{\tilde h}(\tilde\gamma_i)
=
\frac{1}{2\sqrt{\tau_i}}\mathcal L_h(\gamma_i)
=
\ell^h_{(p_0,0)}(q_i,\tau_i)
\leq \frac{n}{2}.
\end{equation}
Since $n\geq 4$, $\mathbb S^{n-1}\times\mathbb R$ is simply connected. Choose an exhaustion $K_1\subset K_2\subset\cdots\subset \mathbb S^{n-1}\times\mathbb R$ by compact connected simply connected domains with smooth boundary, each containing $q_\infty$ in its interior. For each $j$, the Cheeger--Gromov convergence \eqref{eqn:existence-of-asymptotic-shrinker-2} gives, for all sufficiently large $i$, embeddings $\Phi_{i,j}:K_j\to M$ with $\Phi_{i,j}(q_\infty)=q_i$ and $\Phi_{i,j}^*\bigl(\tau_i^{-1}g(\tau_i t)\bigr) \to g_\infty(t)$ smoothly on $K_j\times[-1,-1/2]$. 

Since $K_j$ is simply connected, $\Phi_{i,j}$ lifts uniquely to a map $\widetilde\Phi_{i,j}:K_j\to \widetilde M$ satisfying
$$
\pi\circ \widetilde\Phi_{i,j}=\Phi_{i,j},
\qquad
\widetilde\Phi_{i,j}(q_\infty)=\tilde q_i.
$$
Since $\Phi_{i,j}$ is an embedding and $\pi$ is a local diffeomorphism, the lift $\widetilde\Phi_{i,j}$ is an immersion. It is also injective: if $\widetilde\Phi_{i,j}(a)=\widetilde\Phi_{i,j}(b)$, then $\Phi_{i,j}(a)=\Phi_{i,j}(b)$, and hence $a=b$. Since $K_j$ is compact, $\widetilde\Phi_{i,j}$ is an embedding, i.e., a diffeomorphism onto its image. Finally, using $\tilde g(t)=\pi^*g(t)$, we obtain 
$$
\widetilde\Phi_{i,j}^*
\bigl(\tau_i^{-1}\tilde g(\tau_i t)\bigr)
=
\Phi_{i,j}^*
\bigl(\tau_i^{-1}g(\tau_i t)\bigr)
\to
g_\infty(t)
$$
smoothly on $K_j\times[-1,-1/2]$. Thus
$$
\left(
\widetilde M,
\tilde g_i(t):=\tau_i^{-1}\tilde g(\tau_i t),\,
(\tilde q_i,-1)
\right)
\to
\left(
\mathbb S^{n-1}\times\mathbb R,
g_\infty(t),
(q_\infty,-1)
\right)
$$
in the smooth Cheeger--Gromov sense. Together with \eqref{eqn:proof-of-asymptotic-shrinker-of-univ-cover-1}, this shows that $(\widetilde M,\tilde g(t))_{t\leq 0}$ has $\mathbb S^{n-1}\times\mathbb R$ as an asymptotic shrinker. 
\end{proof}
We use $R(X,Y)Z$ for the curvature endomorphism, $\operatorname{Rm}$ for the corresponding Riemann curvature tensor, $\Ric$ for the Ricci tensor, and $R$ for the scalar curvature. We follow the conventions used by \cite{Bre10}, \cite{Bre19}, and \cite{BN23}. The curvature tensor is defined by 
$$-R(X,Y)Z=\nabla_X \nabla _Y Z-\nabla_Y\nabla_X Z-\nabla_{[X,Y]}Z,$$
for any vector fields $X,Y,Z$, and we write $\langle R(e_i,e_j)e_k,e_l\rangle=R_{ijkl}$ for vectors $e_i,e_j,e_k,e_l$. In this convention, the sectional curvature of the plane spanned by orthonormal 2-frame $\{e_1,e_2\}$ is $R_{1212}$. 
\begin{definition}
\label{def:definitions-of-uniform-PIC}
Let $n\geq 4$ and $\Lambda>0$. 
\begin{itemize}
    \item A Riemannian manifold $(M^n,g)$ is said to have nonnegative isotropic curvature (NIC) if for any $p \in M$ and any orthonormal 4-frame $\left\{e_1, e_2, e_3, e_4\right\}$ of $T_p M$ we have
\begin{equation}
\label{eqn:PIC-inequality-higher-dim}
R_{1313}+R_{1414}+R_{2323}+R_{2424}-2 R_{1234} \geq 0.
\end{equation}
    \item  A Riemannian manifold $(M^n,g)$ is said to have uniformly $\Lambda$-positive isotropic curvature (uniformly $\Lambda$-PIC) if for any $p \in M$ and any orthonormal 4-frame $\left\{e_1, e_2, e_3, e_4\right\}$ of $T_p M$ we have
\begin{equation}
\label{eqn:L-PIC-Pinching-inequality-higher-dim}
R_{1313}+R_{1414}+R_{2323}+R_{2424}-2 R_{1234} \geq 4 \Lambda |\operatorname{Rm}|>0 .
\end{equation}
\item We say a metric $g$ on $M^n$ satisfies the $\Lambda$-PIC pinching inequality at a point $p$ if  (\ref{eqn:L-PIC-Pinching-inequality-higher-dim}) is satisfied for all orthonormal $4$-frames of $T_p M$. 
\item Let $(M^n,g(t))_{t\in [a,b]}$ be a solution to Ricci flow. We say the flow $(M^n,g(t))_{t\in [a,b]}$ is uniformly $\Lambda$-PIC if for each $t\in [a,b]$, $(M^n,g(t))$ is uniformly $\Lambda$-PIC. 
\item Let $(M^n,g(t))_{t\in [a,b]}$ be a solution to Ricci flow. We say the flow $(M^n,g(t))_{t\in [a,b]}$ is uniformly PIC if there exists some $\alpha>0$ for which it is uniformly $\alpha$-PIC. 
\end{itemize}
\end{definition}
\noindent We note the following basic facts about the uniform PIC condition. Let $\Lambda>0$.
\begin{enumerate}[(\roman{enumi})]
    \item (Invariant under scaling) If a Riemannian metric $g$ on a smooth manifold $M$ satisfies $\Lambda$-PIC pinching inequality at a point $p\in M$, then so does $\lambda g$ at $p$ for any $\lambda>0$.
    \item (Preserved under Cheeger--Gromov convergence) Suppose $(M^n_i,g_i,x_i)_{i\geq 1}$ is a sequence of manifolds that converge to $(N,h,y)$ in the smooth Cheeger--Gromov sense. If each $(M^n_i,g_i)$ is uniformly $\Lambda$-PIC, then $(N,h)$ is also uniformly $\Lambda$-PIC. 
    \item  (Stable under closeness) Suppose $(M^n_i,g_i,x_i)_{i\geq 1}$ is a sequence of complete manifolds that converge to a complete manifold $(N,h,y)$ in the smooth Cheeger--Gromov sense. If $(N,h)$ is uniformly $\Lambda$-PIC and $D>0$, there exists $i_1$ (depending on $D$) such that for all $i\geq i_1$, every point of $B_{g_i}[x_i;D]\sub M^n_i$ satisfies $(\Lambda/2)$-PIC pinching inequality with respect to $g_i$. 
\end{enumerate}
We refer the reader to \cite{Ham97} and \cite{Bre10} for more details about PIC. 

The following lemma fixes a constant $\theta_n$ such that round cylinder and Bryant soliton satisfy the PIC pinching inequality with constant $\theta_n$. This will be used throughout the paper. 
\begin{lemma}
\label{lem:uniform-theta}
There exists $\theta_n>0$ depending only on $n$ such that the round cylinder $\sph^{n-1}\times \R$ and the Bryant soliton $\operatorname{Bry}^n$ are both uniformly $\theta_n$-PIC. 
\end{lemma}
\begin{proof}
Let $M=\sph^{n-1}\times\R$ with the standard product metric $\bar g$, normalized so that $|\operatorname{Rm}(\bar g)|=1$. $\bar g$ is uniformly $\theta_{cyl}$-PIC for some $\theta_{cyl}>0$ depending only on $n$. Since any standard metric on $\sph^{n-1}\times\R$ differs from $\bar g$ by a rescaling and an isometry, and the uniform PIC condition is scale-invariant, every standard round cylinder is uniformly $\theta_{cyl}$-PIC.

Next, we fix a Bryant soliton metric $g$. For each $\varepsilon>0$, there exists a compact subset $K\sub \Bry^n$ such that every point of $\Bry^n\setminus K$ is the center of an $\varepsilon$-neck, with respect to $g$. Choosing $\varepsilon$ small enough, it follows that every point of $\Bry^n\setminus K$ satisfies $\theta_{cyl}/2$-PIC pinching inequality. Because $g$ has positive curvature operator, there exists $c>0$ such that every point in $K$ satisfies $c$-PIC pinching inequality. As a result, $(\Bry^n,g)$ is uniformly $\theta_{Bry}$-PIC where $\theta_{Bry}=\min(\theta_{cyl}/2,c)$. Any Bryant soliton metric $h$ differs from $g$ by a rescaling and an isometry, hence $h$ is also uniformly $\theta_{Bry}$-PIC. We may choose $\theta_n:=\min(\theta_{cyl},\theta_{Bry})$ to finish the proof.  
\end{proof}
We now prove a propagation property for the uniform PIC pinching condition, which may be of independent interest. 
\begin{proposition}
\label{prop:continuity-theta}
Given $n\geq 4,\theta >0$ and $C_0>0$, there exists $c_n\in (0,1)$, depending only on $n$, and $\delta>0$ depending on $\theta,C_0,n$, such that the following is true. Let $T>\delta$ and $(M^n,g(t))_{t\in [0,T]}$ be any complete Ricci flow with nonnegative curvature operator such that $0\leq R\leq C_0$ on $M\times [0,T]$. If $(M^n,g(0))$ is uniformly $\theta$-PIC, then $(M^n,g(t))_{t\in [0,\delta]}$ is uniformly $\theta c_n$-PIC.  
\end{proposition}
\begin{proof}
Consider $\R^n$ with its Euclidean inner product. Let $ \mathcal{C}_B(\R^n)$ denote the space of all algebraic curvature tensors on $\R^n$. For $S\in \mathcal{C}_B(\R^n)$ we denote by $\Ric(S)$ and $\operatorname{scal}(S)$ the Ricci tensor and scalar curvature of $S$. We also consider Hamilton's ODE $\frac{d}{dt} S=Q(S):=[S^2+S^{\#}]$. Under Hamilton's ODE, 
$$\frac{d}{dt}\operatorname{scal}(S(t))=2|\Ric(S(t))|^2.$$
 Define 
$$C_{NIC}(\R^n):=\{S\in \mathcal{C}_B(\R^n):S\text{ has NIC}\}.$$
We define $I:=\frac{1}{2}\mathrm{id}\odot\mathrm{id}$ which corresponds to the identity curvature operator (with isotropic curvature 4) and $\odot$ is the Kulkarni-Nomizu product. We let $c_n\in (0,\frac{1}{2})$, depending only on $n$, such that $|S|\geq 2c_n|\operatorname{scal}(S)|$ for all $S\in \mathcal{C}_B(\R^n)$. Note that if $S\in \mathcal{C}_B(\R^n)$ with nonnegative curvature operator, and if $S-\beta \operatorname{scal}(S)I\in C_{NIC}$ then $S$ is uniformly $\beta$-PIC because $\operatorname{scal}(S)\geq |S|$. On the other hand, if $S$ is uniformly $\beta$-PIC, then $S-2c_n\beta \operatorname{scal}(S)I\in C_{NIC}$. 

\noindent \textbf{Case 1. }Suppose now that $\beta\geq \frac{1}{(n-1)}$. Suppose that $S_0-\beta I\operatorname{scal}(S_0)\in C_{NIC}$, $S_0$ has nonnegative curvature operator, and $S(t)$ is a solution to Hamilton's ODE starting at $S_0$. Then, $S(t)$ has nonnegative curvature operator for all $t$. Define $\kappa(t):=\beta \operatorname{scal}(S(t))$. We observe that 
$$\begin{aligned}
\kappa'(t)-2(n-1)\kappa(t)^2&=2\beta|\Ric|^2-2(n-1)\beta^2\operatorname{scal}^2(S(t))\\
&\leq 2\beta \operatorname{scal}(S(t))^2(1-(n-1)\beta)\\ 
& \leq 0
\end{aligned}$$
where we used $|\Ric(S(t))|\leq \operatorname{scal}(S(t))$. By \cite[Proposition A.5]{Bre19}, $S(t)-\kappa(t)I\in C_{NIC}$. Hence $S(t)-\beta \operatorname{scal}(S(t))I\in C_{NIC}$ implying $S(t)$ is uniformly $\beta$-PIC. Suppose $(M^n,g(t))_{t\in [0,T]}$ is a solution to Ricci flow with bounded nonnegative curvature and $g(0)$ is uniformly ${\theta}$-PIC with ${\theta}\geq \frac{1}{2c_n(n-1)}$, then $\operatorname{Rm}_{g(0)}-\beta R_{g(0)}I\in C_{NIC}$ where $\beta:=2c_n\theta \geq \frac{1}{n-1}$. By applying maximum principles, it follows that $\operatorname{Rm}_{g(t)}-\beta R_{g(t)}I\in C_{NIC}$ for all $t\in [0,T]$ hence $g(t)$ is uniformly $\beta$-PIC for all $t$. Since $\beta\geq c_n\theta$, we obtain that $g(t)$ is uniformly $\theta c_n$-PIC for all $t$. This proves the proposition when $\theta \geq \frac{1}{2c_n (n-1)}$. 

\noindent \textbf{Case 2. }Let $\beta \in (0,\frac{1}{n-1})$. Suppose that $S_0$ has nonnegative curvature operator, and $S(t)$ is a solution to Hamilton's ODE starting at $S_0$ such that $\operatorname{scal}(S(t))\leq C_0$ for all $t\in[0,T]$. Note that $S(t)$ has nonnegative curvature operator for all $t$. Define $\kappa(t):= b(t) \operatorname{scal}(S(t))$, where
\[
b(t):=\frac{\beta e^{-2 C_0 t}}{1-(n-1) \beta\left(1-e^{-2 C_0 t}\right)}.
\]
Note that $b(t)-(n-1)b(t)^2>0$ and $b'(t)+2(b(t)-(n-1)b(t)^2)C_0\equiv 0$ for all $t>0$. Further, 
$$\begin{aligned}
\kappa'(t)-2(n-1)\kappa(t)^2&=b'(t)\operatorname{scal}(S(t))+2b(t)|\Ric(S(t))|^2-2(n-1)b(t)^2\operatorname{scal}(S(t))^2\\
&\leq [b'(t)+2(b(t)-(n-1)b(t)^2)\operatorname{scal}(S(t))]\operatorname{scal}(S(t))\\
&\leq [b'(t)+2(b(t)-(n-1)b(t)^2)C_0]\operatorname{scal}(S(t))\\
&\leq 0
\end{aligned}$$
where we used $|\Ric(S)|\leq \operatorname{scal}(S)$. Therefore, by \cite[Proposition~A.5]{Bre19}, if $S(t_0)-\kappa(t_0)I\in C_{NIC}$ for some $t_0\in[0,T]$, then $S(t)-\kappa(t)I\in C_{NIC}$ for all $t\in[t_0,T]$.

Assume now that $(M^n,g(t))_{t\in [0,T]}$ is a complete Ricci flow with nonnegative curvature operator such that $0\leq R\leq C_0$ on $M\times [0,T]$. Then, the curvature operator is bounded on $M\times [0,T]$. Let $\theta\in (0,\frac{1}{2c_n (n-1)})$. Set $\beta:=2\theta c_n\in (0,\frac{1}{n-1})$, $\bar{C}_0=2C_0$ and $\bar{b}(t):=\frac{\beta e^{-2 \bar{C}_0 t}}{1-(n-1) \beta\left(1-e^{-2 \bar{C}_0 t}\right)}$ for $t>0$. Choose $\delta>0$ (depending only on $n,\theta$ and $C_0$) such that $\bar{b}(t)\geq \beta/2$ for all $t\in [0,\delta]$. Suppose $(M^n,g(0))$ is uniformly $\theta$-PIC, so that $\operatorname{Rm}_{g(0)}-\beta R_{g(0)}I\in C_{NIC}$. 

We now apply the maximum principle with avoidance sets \cite[Theorem~12.38]{CCG+08}. 
Let $\mathcal V\to M$ be the vector bundle whose fiber at $p\in M$ is the space
of algebraic curvature tensors on $T_pM$. 
Define 
$$K(t):=\{S\in \mathcal{C}_B(\R^n):S-\bar{b}(t)\operatorname{scal}(S)I\in C_{NIC},S\gtrsim 0\}$$
and $A:=\{S\in \mathcal{C}_B(\R^n):\operatorname{scal}(S)\geq \bar{C}_0\}$ (where $\gtrsim 0$ denotes the nonnegativity of curvature operator). Define $\mathcal{K}(t)\subset \mathcal V$ as follows: an element $S\in \mathcal V_p$ belongs to $\mathcal K(t)_p$ if and only if for some (equivalently, any) linear isometry $\eta_{p,t}:(T_pM,g(t))\to \mathbb R^n$, one has $(\eta_{p,t})_*S\in K(t)$. 
$\mathcal{K}(t)$ is well-defined and invariant under parallel translation as $K(t)$ is $O(n)$-invariant. 
The fibers $\mathcal{K}(t)_p,p\in M$ are closed and convex as $K(t)$ is closed and convex. The spacetime track $$
\mathcal{T} :=\{(v, t) \in \mathcal{V} \times \mathbb{R}: v \in \mathcal{K}(t), 0 \leq t \leq T\}
$$
is closed in $\mathcal{V} \times[0, T]$. Define the avoidance set $\mathcal{A}(t)$ similarly: for each $p\in M$, an element $S\in \mathcal V_p$ belongs to $\mathcal A(t)_p$ if and only if, for some (equivalently, any) linear isometry $\eta_{p,t}:(T_pM,g(t))\to \R^n$, one has $(\eta_{p,t})_*S\in A$.
The avoidance track 
$$\mathcal{A}\mathcal{T}:=\{(v, t) \in \mathcal{V} \times \mathbb{R}: v \in \mathcal{A}(t), 0 \leq t \leq T\}$$
is also closed in $\mathcal{V}\times [0,T]$. 

By the discussion above, we have that for any $x \in {M}$ and initial time $t_0 \in[0, T)$, if a solution $U(t)$ of Hamilton's ODE starts at $U_0\in \mathcal{K}(t_0)_x\setminus \mathcal{A}(t_0)_x$, then either $U(t)\in \mathcal{K}(t)_x$ for all $t\geq t_0$, or there exists $t_1\geq t_0$ such that $U(t)\in \mathcal{K}(t)_x\setminus \mathcal{A}(t)_x$ for all $t\in [t_0,t_1]$ and $U(t_1)\in \mathcal{A}(t_1)_x$. 

The hypotheses of \cite[Theorem~12.38]{CCG+08} are now satisfied, so $\operatorname{Rm}_{g(t)}\in \mathcal{K}(t)$ for all $t\in [0,T]$, i.e. $\operatorname{Rm}_{g(t)}-\bar{b}(t) R_{g(t)}I\in C_{NIC}$ for $t\in [0,T]$. Hence, $(M^n,g(t))$ is uniformly $\bar{b}(t)$-PIC for all $t\in [0,T]$. As a result, $(M^n,g(t))$ is uniformly $\beta/2$-PIC for all $t\in [0,\delta]$. This proves the proposition when $0<\theta\leq \frac{1}{2c_n (n-1)}$. 
\end{proof}
The family of $\kappa$-solutions with $\sph^{n-1}\times \R$ as asymptotic shrinker enjoys the following compactness property, see \cite[Theorem~1.3]{LZ22}. This will be used repeatedly in the proof of Theorem~\ref{thm:full-classification}. 
\begin{lemma}
\label{lem:compactness-result}
Let $n\geq 4$. Let $\left(M_k^n, g_k(t), p_k\right)_{t \in(-\infty, 0]}$, $k\geq 1$, be a sequence of $n$-dimensional $\kappa$-solutions each having $\sph^{n-1}\times \R$ as one of the asymptotic shrinkers. Let $Q_k=R_k\left(p_k, 0\right)$ and $\bar{g}_k(t)=Q_k g_k\left(t Q_k^{-1}\right)$. Then after possibly passing to a subsequence, 
$$\left(M_k^n, \bar{g}_k(t), p_k\right)_{t \in(-\infty, 0]}\to \left(M^n_{\infty}, g_{\infty}(t), p_{\infty}\right)_{t \in(-\infty, 0]},$$
in the smooth Cheeger--Gromov sense where $\left(M^n_{\infty}, g_{\infty}(t), p_{\infty}\right)_{t \in(-\infty, 0]}$ is a $\kappa$-solution that satisfies the following condition: either (i)  $\left(M^n_{\infty}, g_{\infty}(t)\right)_{t \in(-\infty, 0]}$ is the standard shrinking sphere $\mathbb{S}^n$ or (ii) the asymptotic shrinker based at any point in $M^n_{\infty} \times(-\infty, 0]$ is $\mathbb{S}^{n-1} \times \mathbb{R}$. In particular, if the limit $M^n_\infty$ splits, it must be $\mathbb{S}^{n-1} \times \mathbb{R}$. 
\end{lemma}
As remarked in \cite{LZ22}, the fact that $g_\infty(t)$ has bounded curvature follows from the assumption that each $M^n_k$ has $\sph^{n-1}\times \R$ as one of the asymptotic shrinkers. 

We end this section with an application of Hamilton's strong maximum principle.
\begin{lemma}
\label{lem:strong-max-principle-line-or-positive-curvature}
Let $(M^n,g(t))_{t\leq 0},n\geq 4$ be a $\kappa$-solution with asymptotic shrinker $\sph^{n-1}\times \R$. Then either $(M,g(t))_{t\leq 0}$ has positive curvature operator, or $(M,g(t))_{t\leq 0}$ is the round cylinder $\sph^{n-1}\times \R$.
\end{lemma}
\begin{proof}
Let $\pi:\widetilde M\to M$ be the universal cover, and let $\tilde g(t):=\pi^*g(t)$. By Lemma~\ref{lem:asymptotic-shrinker-of-universal-cover}, $(\widetilde M,\tilde g(t))_{t\le 0}$ is a $\kappa$-solution whose asymptotic shrinker is $\sph^{n-1}\times \R$. If $(\widetilde M,\tilde g(t))_{t\le 0}$ has positive curvature operator, then so does $(M,g(t))_{t\le 0}$, since $\pi$ is a local isometry for each $t$.

Suppose that $(\widetilde M,\tilde g(t))_{t\le 0}$ fails to have positive curvature operator. Fix $t_0<0$. Since $(\widetilde M,\tilde g(t))$ is simply connected, complete, and has nonnegative curvature operator, \cite[Exercise~7.36]{CLN06} implies that $(\widetilde M,\tilde g(t_0))$ splits isometrically as
\begin{equation}
\label{eqn:1-strong-max-principle-line-or-positive-curvature}
(\widetilde M,\tilde g(t_0)) \cong \R^k \times N_1\times \cdots \times N_r,
\end{equation}
where each $N_i$ is either an irreducible nonflat closed Einstein symmetric space with $\operatorname{Rm}\ge 0$, a manifold with positive curvature operator, or a K\"ahler manifold with positive curvature operator on real $(1,1)$-forms. By uniqueness for the Ricci flow on complete bounded-curvature manifolds, this product decomposition persists for $t\in [t_0,0]$, and since $t_0<0$ was arbitrary, for all $t\le 0$.

If the product decomposition \eqref{eqn:1-strong-max-principle-line-or-positive-curvature} has only one factor, then there are two possibilities. First, $(\widetilde M,\tilde g(t_0))$ could be an irreducible closed Einstein symmetric space. In that case it follows that $(\widetilde M,\tilde g(t))_{t\le 0}$ is a shrinking Einstein soliton. Hence its asymptotic shrinker is the same compact Einstein symmetric space, not $\sph^{n-1}\times\R$, a contradiction. Second, $(\widetilde M,\tilde g(t))_{t\leq 0}$ could be K\"ahler. Since the K\"ahler condition is preserved by the Ricci flow, every asymptotic shrinker would then be K\"ahler, which is impossible because the round cylinder $\sph^{n-1}\times\R$ is not K\"ahler for $n\ge 4$. 

Thus, $(\widetilde M,\tilde g(t_0))$ splits nontrivially as a Riemannian product. As noted in \cite{LZ22}, if a $\kappa$-solution with asymptotic shrinker $\sph^{n-1}\times \R$ splits, then it must be the round cylinder $\sph^{n-1}\times \R$. Hence $(\widetilde M,\tilde g(t))$ is isometric to $\sph^{n-1}\times \R$ showing that $(M,g(t))$ is a quotient of $\sph^{n-1}\times \R$. Since the asymptotic shrinker of $(M,g(t))_{t\leq 0}$ is $\sph^{n-1}\times \R$, the quotient must be trivial. Therefore $(M,g(t))$ is isometric to the round cylinder $\sph^{n-1}\times \R$.
\end{proof}

\section{Noncompact \texorpdfstring{$\kappa$}{kappa}-solutions}
\label{sec:noncompact-case}
Let $n\geq 4$. In this section we treat noncompact $\kappa$-solutions with asymptotic shrinker $\sph^{n-1}\times\R$. We prove that such solutions are uniformly PIC, thereby establishing Theorem~\ref{thm:main-result} in the noncompact case. 

We have the following cap-neck decomposition theorem due to \cite{LZ22}. This implies that the non-neck-like region has bounded geometry.  
\begin{lemma}
\label{lem:can-neighborhood}
There exists $\varepsilon_0>0$ such that for every $\varepsilon\in(0,\varepsilon_0)$, there exists $C=C(n,\varepsilon,\kappa)<\infty$ with the following property. Let $(M^n,g(t))_{t\leq 0}$ be a noncompact $\kappa$-solution with positive curvature operator, and with asymptotic shrinker $\sph^{n-1}\times \R$. For $t\leq 0$, define $$\Om_\varepsilon(t)=\{x\in M^n:x \text{ is not an }\varepsilon\text{-neck in }(M^n,g(t))\}.$$ Then, for all $t\leq 0$, $\Om_\varepsilon(t)$ is nonempty, compact, and 
\begin{equation}
\label{eqn:can-neighborhood-cap-diameter-estimate-l}
\sup_{o\in \Om_\varepsilon(t)}R(o,t)\operatorname{diam}_{g(t)}(\Om_\varepsilon(t))^2\leq C.
\end{equation}
\end{lemma}

\begin{proof}
We first show that $\Omega_\varepsilon(t)$ is nonempty. Suppose instead that $\Omega_\varepsilon(t_0)=\emptyset$ for some $t_0\le 0$. Then every point of $(M,g(t_0))$ is the center of an $\varepsilon$-neck. If $\varepsilon<\varepsilon_0(n)$ is sufficiently small, Hamilton's constant-mean-curvature (CMC) foliation implies that $M$ is foliated by leaves diffeomorphic to $\sph^{n-1}$. Since $M$ is complete and noncompact, it follows that $M$ is diffeomorphic to $\sph^{n-1}\times \R$, and in particular has two ends. Because $\Ric_{g(t_0)}\ge 0$, the Cheeger--Gromoll splitting theorem implies that $(M,g(t_0))$ splits off a line isometrically. This contradicts the assumption that $(M,g(t_0))$ has positive curvature operator.

The remaining assertions follow from \cite[Theorem 1.4]{LZ22}.
\end{proof}
We next show that the caps in $M^n$ must look like the Bryant soliton at certain large negative times after suitable rescaling. 
\begin{lemma}
\label{lem:type-II}
Let $n\geq 4$. Let $(M^n,g(t))_{t\leq 0}$ be a noncompact $\kappa$-solution, with positive curvature operator, and with asymptotic shrinker $\sph^{n-1}\times \R$. There exists $\varepsilon_1>0$ and a sequence of spacetime points $(\hat{p}_k,\hat{t}_k)$ such that $\hat{p}_k\in \Om_\varepsilon(\hat{t}_k)$ for all $\varepsilon\in (0,\varepsilon_1)$, $\hat{t}_k\to -\infty$, and 
\begin{equation}
\label{eqn:convergence-to-Bryn}
\left(M^n,R(\hat{p}_k,\hat{t}_k)g\left(\frac{t}{R(\hat{p}_k,\hat{t}_k)}+\hat{t}_k\right),\hat{p}_k\right)_{t\in (-\infty,\infty)}\to (\operatorname{Bry}^n,g_\infty(t),\hat{x})_{t\in (-\infty,\infty)},
\end{equation}
where the convergence is in the smooth pointed Cheeger--Gromov sense on compact time intervals in $(-\infty,\infty)$, $g_\infty(t)$ is the canonical Ricci flow on $\operatorname{Bry}^n$, and $\hat{x}$ is the tip of $\operatorname{Bry}^n$. 
\end{lemma}
\begin{proof}
Since $M$ has positive sectional curvature, it follows from \cite[Corollary~1.2]{LR24} that the solution is Type II: 
\begin{equation}
\label{eqn:-type-II-equation}
\sup_{(x,t)\in M\times(-\infty,0]} (-t)\,R(x,t)=\infty.
\end{equation}
Following exactly as in \cite[Proposition~9.7]{Bre20}, we may choose points $\left.(\hat{p}_k, \hat{t}_k\right.) \in M \times(-k, 0)$, such that $\hat{t}_k\to -\infty$, $\left.(-\hat{t}_k\right.) R\left.(\hat{p}_k, \hat{t}_k\right.) \rightarrow \infty$ and
$$
\left(M^n,\;R(\hat{p}_k,\hat{t}_k)\,g\!\left(\frac{t}{R(\hat{p}_k,\hat{t}_k)}+\hat{t}_k\right),\;\hat{p}_k\right)
\to
\left(N^n,\;g_\infty(t),\;\hat{x}\right)_{t\in \mathbb{R}},
$$
where $(N^n,g_\infty(t))$ is an eternal $\kappa$-solution such that $R_{g_\infty(t)}\leq 1$ on $N^n$ for all $t\in \R$ and $R_{g_{\infty}(0)}(\hat{x})=1$. On the other hand, Lemma~\ref{lem:compactness-result} shows that either (i) $(N^n,g_\infty(t))$ is isometric to standard shrinking sphere $\sph^n$ or (ii) an asymptotic shrinker of $(N^n,g_\infty(t))$ is $\sph^{n-1}\times \R$. Since each rescaled flow is defined on the complete, connected, noncompact manifold $M^n$, any pointed Cheeger--Gromov limit is again noncompact. Therefore, case (i) cannot occur. 

If $N^n$ fails to have positive curvature operator, then by Lemma~\ref{lem:strong-max-principle-line-or-positive-curvature}, $N^n$ is isometric to the standard shrinking cylinder $\sph^{n-1}\times \R$, which is not eternal. This shows that $(N^n,g_\infty(t))$ has positive curvature operator. By \cite{Ham93}, $N^n$ is a steady soliton. By \cite[Theorem~5.2]{CMZ25a} (which is based on \cite{Bre14} and classifies positively curved steady solitons with $\sph^{n-1}\times \R$ as the asymptotic shrinker), it follows that $N^n$ is isometric to Bryant soliton $\operatorname{Bry}^n$. As $\hat{x}$ corresponds to the maximum point of scalar curvature in space, $\hat{x}$ is the tip of the Bryant soliton. Hence, there exists $\varepsilon_1>0$ such that for every $\varepsilon \in\left(0, \varepsilon_1\right), \hat{p}_k \in \Omega_{\varepsilon}(\hat{t}_k)$ for all large $k$. 
\end{proof}
The previous lemma identifies the limit of the rescaled flows based at $(\hat p_k,\hat t_k)$. We next classify all possible limits of the rescaled flows at times $\hat{t}_k$. 
\begin{lemma}
\label{lem:limits-at-hat-t-k}
Let $n\geq 4$. Let $(M^n,g(t))_{t\leq 0}$ be a noncompact $\kappa$-solution, with positive curvature operator, and with asymptotic shrinker $\sph^{n-1}\times \R$. Let $\hat{t}_k$ be as in Lemma~\ref{lem:type-II}. Given any $q_k\in M$, we may pass to a subsequence to ensure  
\begin{equation}
\label{eqn:limits-at-hat-t-k}
\left(M^n,g_k(t):=R\left(q_k,\hat{t}_k\right)g\left(\frac{t}{R(q_k,\hat{t}_k)}+\hat{t}_k\right),q_k\right)_{t\in (-\infty,0]}\to \left(N^n,g_\infty(t),x_\infty\right)_{t\in (-\infty,0]},
\end{equation}
in the smooth Cheeger--Gromov sense, where $(N^n,g_\infty(t))$ is either the Bryant soliton $\operatorname{Bry}^n$, or the round shrinking cylinder $\sph^{n-1}\times \R$. 
\end{lemma}
\begin{proof}
By Lemma~\ref{lem:compactness-result}, after passing to a subsequence,
\[
(M^n,g_k(t),q_k)_{t\le 0}\to (N^n,g_\infty(t),x_\infty)_{t\le 0},
\]
in the smooth pointed Cheeger--Gromov sense, where $(N^n,g_\infty(t))$ is a $\kappa$-solution. Recall $\varepsilon_0$ and the points $\hat p_k$ from Lemma~\ref{lem:type-II}.

\noindent \textbf{Case 1. }Suppose that after passing to a subsequence, $\sup_k R(\hat p_k,\hat t_k)\,d_{g(\hat t_k)}(\hat p_k,q_k)^2<\infty.$ Set $\tilde g_k(t):=R(\hat p_k,\hat t_k)\,g\!\left(\hat t_k+\frac{t}{R(\hat p_k,\hat t_k)}\right)$. By \eqref{eqn:convergence-to-Bryn},
\[
(M^n,\tilde g_k(t),\hat p_k)_{t\in(-\infty,\infty)}
\to
(\Bry^n,g_{\Bry}(t),\hat x)_{t\in(-\infty,\infty)}.
\]
Since the points $q_k$ stay in a bounded $\tilde g_k(0)$-distance from $\hat p_k$, after passing to a subsequence, there exists $w\in \Bry^n$ such that
\[
(M^n,\tilde g_k(0),q_k)
\to
(\Bry^n,g_{\Bry}(0),w)
\]
in the smooth Cheeger--Gromov sense. In particular,
\[
\alpha_k:=\frac{R(q_k,\hat t_k)}{R(\hat p_k,\hat t_k)}
=R_{\tilde g_k(0)}(q_k)\to R_{g_{\Bry}(0)}(w)>0.
\]
Since $g_k(0)=\alpha_k\,\tilde g_k(0),$ it follows that
\[
(M^n, g_k(0),q_k)
\to
(\Bry^n,\alpha g_{\Bry}(0),w),
\]
where $\alpha=R_{g_{\Bry}(0)}(w)>0$. Therefore, $(N^n,g_\infty(0),x_\infty)$ is isometric to $(\Bry^n,\alpha g_{\Bry}(0),w)$. By backward uniqueness \cite{Kot10}, the entire flow $(N^n,g_\infty(t))_{t\in(-\infty,0]}$ is the canonical Ricci flow on the Bryant soliton. 

\noindent \textbf{Case 2. }Suppose that $R(\hat p_k,\hat t_k)\,d_{g(\hat t_k)}(\hat p_k,q_k)^2\to \infty.$ We claim that for every $\delta\in(0,\varepsilon_0)$, the point $(q_k,\hat t_k)$ is the center of a $\delta$-neck for all sufficiently large $k$. Suppose not. Then, after passing to a subsequence, there exists $\delta_1\in(0,\varepsilon_0)$ such that $q_k\in \Omega_{\delta_1}(\hat t_k)$ for all $k$. Since also $\hat p_k\in \Omega_{\delta_1}(\hat t_k)$ for all large $k$, the estimate in Lemma~\ref{lem:can-neighborhood} gives
\[
R(\hat p_k,\hat t_k)\,
d_{g(\hat t_k)}(\hat p_k,q_k)^2
\le
R(\hat p_k,\hat t_k)\,
\operatorname{diam}_{g(\hat t_k)}(\Omega_{\delta_1}(\hat t_k))^2
\le C_{\delta_1},
\]
a contradiction.

Therefore, for each $\delta\in(0,\varepsilon_0)$, after passing to a subsequence, $(q_k,\hat t_k)$ is the center of a $\delta$-neck for all sufficiently large $k$. Hence $(M^n,g_k(0),q_k)$ converges to $ \sph^{n-1}\times\R$ in the smooth Cheeger--Gromov sense. Therefore $(N^n,g_\infty(0),x_\infty)$ is isometric to the round cylinder. By backward uniqueness \cite{Kot10}, the entire flow $(N^n,g_\infty(t))_{t\in(-\infty,0]}$ is the shrinking round cylinder.

This proves the lemma.
\end{proof}
We now prove the main result of this section where we show that a noncompact $\kappa$-solution with asymptotic shrinker $\sph^{n-1}\times \R$ is uniformly PIC. This proves Theorem~\ref{thm:main-result} in the noncompact case. 
\begin{theorem}
\label{thm:noncompact-kappa-has-PIC}
Let $(M^n,g(t))_{t\leq 0},n\geq 4$ be a noncompact $\kappa$-solution with asymptotic shrinker $\sph^{n-1}\times \R$,  with positive curvature operator. Then, there exists $\Lambda>0,b<0$ such that $(M^n,g(t))_{t\leq b}$ is uniformly $\Lambda$-PIC. 
\end{theorem}

\begin{proof}
Let $C_0>0$ be such that $\sup_{(x,t)\in M\times (-\infty,0]}R(x,t)\leq C_0$. Recall the choice of $\varepsilon>0$ and the points $(\hat{p}_k,\hat{t}_k)$ from Lemma~\ref{lem:type-II}. Let $\theta_n$ be the constant from Lemma~\ref{lem:uniform-theta}. Let $c_n\in (0,1)$ be the constant from Proposition~\ref{prop:continuity-theta}. 

By taking a smaller $\varepsilon>0$, we may ensure that the $\varepsilon$-neck definition gives $C^m$-closeness (after scaling) to the round cylinder with $m$ large, so that every point which is the center of an $\varepsilon$-neck satisfies $\theta_n/2$-PIC pinching inequality. We further choose $\varepsilon$ small enough such that the neck-stability theorem in \cite[Theorem 3.11]{LZ22} applies for this choice of $\varepsilon$. 

Let $t\leq 0$. Since every point in $(M^n\setminus\Om_\varepsilon(t),g(t))$ is an $\varepsilon$-neck, it follows that every point in $(M^n\setminus\Om_\varepsilon(t),g(t))$ satisfies the $\theta_n/2$-PIC pinching inequality. 
\begin{claim}
\label{claim:claim-1-noncompact-kappa-has-PIC}
For each $L>0$, the flow $(M^n,g(t))_{t\in [\hat{t}_k-LR(\hat{p}_k,\hat{t}_k)^{-1},\hat{t}_k+LR(\hat{p}_k,\hat{t}_k)^{-1}]}$ is uniformly $\theta_n/2$-PIC for all large $k$ (depending on $L$). 
\end{claim}
\begin{proof}
Under the convergence \eqref{eqn:convergence-to-Bryn},
\[
\left(M^n,\;R(\hat p_k,\hat t_k)\,g\!\left(\hat t_k+\frac{s}{R(\hat p_k,\hat t_k)}\right),\;\hat p_k\right)_{s\in[-L,L]}
\to
(\Bry^n,g_\infty(s),\hat x)_{s\in[-L,L]}
\]
smoothly in the pointed Cheeger--Gromov sense. Since the Bryant tip is not the center of an $\varepsilon$-neck, it follows that for all large $k$ the point $\hat p_k$ is not the center of an $\varepsilon$-neck at time $\hat t_k+sR(\hat p_k,\hat t_k)^{-1}$ for every $|s|\le L$, i.e.
\[
\hat p_k\in \Omega_\varepsilon\!\left(\hat t_k+sR(\hat p_k,\hat t_k)^{-1}\right)
\qquad\text{for all }|s|\le L.
\]
Moreover, by smooth convergence on $[-L,L]$, there exists $A=A(L)\ge 1$ such that for all large $k$ and all $|s|\le L$,
\[
A^{-1}R(\hat p_k,\hat t_k)
\le
R\!\left(\hat p_k,\hat t_k+sR(\hat p_k,\hat t_k)^{-1}\right)
\le
A\,R(\hat p_k,\hat t_k).
\]
Combining this with the cap-diameter \eqref{eqn:can-neighborhood-cap-diameter-estimate-l} estimate yields
\[
R(\hat p_k,\hat t_k)\,
\operatorname{diam}_{g(\hat t_k+sR(\hat p_k,\hat t_k)^{-1})}
\!\left(\Omega_\varepsilon\!\left(\hat t_k+sR(\hat p_k,\hat t_k)^{-1}\right)\right)^2
\le A\,C_\varepsilon
\]
for all $|s|\le L$ and all large $k$. It follows that $\Omega_\varepsilon\!\left(\hat t_k+sR(\hat p_k,\hat t_k)^{-1}\right)$ is contained in a ball of radius at most $(A\,C_\varepsilon)^{1/2}$ about $\hat p_k$ with respect to the rescaled metric $R(\hat p_k,\hat t_k)\,g\!\left(\hat t_k+sR(\hat p_k,\hat t_k)^{-1}\right).$ Therefore, for each fixed $L>0$, the regions
\[
\Omega_\varepsilon\!\left(\hat t_k+sR(\hat p_k,\hat t_k)^{-1}\right),\qquad |s|\le L,
\]
lie in a uniformly bounded part of the rescaled flow, and hence converge smoothly to the corresponding compact regions of the Bryant soliton. Since the Bryant soliton is uniformly $\theta_n$-PIC, we conclude that every point of $\Omega_\varepsilon\!\left(\hat t_k+sR(\hat p_k,\hat t_k)^{-1}\right)$ satisfies the $\theta_n/2$-PIC pinching inequality for all $|s|\le L$ and all sufficiently large $k$. Together with the fact that every point of the complement is the center of an $\varepsilon$-neck and hence also satisfies the $\theta_n/2$-PIC pinching inequality, this proves the claim.
\end{proof}
Suppose, for contradiction, that the conclusion of the theorem fails. Then for every $\beta>0$ and $b_0<0$, the flow $(M^n,g(t))_{t\leq b_0}$ is not uniformly $\beta$-PIC. Pick $\alpha>0$ such that $\alpha<\theta_n c_n/100$. 

Define for all large $k$, 
$$t_k:=\inf\{t\in [\hat{t}_k,0]:(M,g(t))\text{ is not uniformly }\alpha\text{-PIC}\}.$$
Applying the contradiction assumption with $\beta:=\alpha$ and $b_0:=-1$, we can fix $t_* \leq-1$ such that $(M, g\left(t_*\right))$ is not uniformly $\alpha$-PIC. Since $\hat{t}_k \rightarrow-\infty$, for all large $k$ we have $\hat{t}_k \leq t_*$, hence $t_* \in \left[\hat{t}_k, 0\right]$. Therefore the defining set is nonempty and $t_k$ is well-defined for all large $k$. 

As $(M,g(\hat{t}_k))$ is $\alpha$-PIC, it follows that $t_k\in (\hat{t}_k,0]$ and that the flow $(M,g(t))_{t\in [\hat{t}_k,t_k)}$ is uniformly $\alpha$-PIC. By Claim~\ref{claim:claim-1-noncompact-kappa-has-PIC} and the choice $\alpha<\theta_n/2$, we obtain that for each $L>0$, $t_k>\hat{t}_k+LR(\hat{p}_k,\hat{t}_k)^{-1}$ for all large $k$ (depending on $L$). 

If $t_{\mathrm{inf}}:=\inf_k t_k>-\infty$, then $(M,g(t))_{t<t_{\mathrm{inf}}}$ is uniformly $\alpha$-PIC which contradicts our assumption. Hence after passing to a subsequence, we may ensure $t_k\to -\infty$. 

\begin{claim}
\label{claim:claim-2-noncompact-kappa-has-PIC}
 $\hat{p}_k\in \Om_\varepsilon(t_k)$ for all large $k$. 
\end{claim}
\begin{proof}
Suppose not, i.e. after passing to a subsequence, $(\hat{p}_k,t_k)$ is the center of an $\varepsilon$-neck. Let $0<\delta\leq \varepsilon$.  Consider the parabolically rescaled flow 
\[
\tilde g_k(t):=R(\hat p_k,t_k)\,g\!\left(t_k+\frac{t}{R(\hat p_k,t_k)}\right),
\]
so that $R_{\tilde g_k(0)}(\hat p_k)=1$ and $\hat p_k$ is the center of an $\varepsilon$-neck in $(M,\tilde{g}_k(0))$. By the neck stability theorem \cite[Theorem 3.11]{LZ22}, there exists $T=T(\delta,\varepsilon,\kappa)>0$ such that $\hat p_k$ is the center of a $\delta$-neck in $(M,\tilde g_k(t))$ for all $t\le -T$, that is, $\hat p_k$ is the center of a $\delta$-neck in $(M,g(s))_{s\leq t_k-TR(\hat{p}_k,t_k)^{-1}}$. By Hamilton's trace Harnack inequality $R(\hat{p}_k,\hat{t}_k)\leq R(\hat{p}_k,t_k)$ for all $k$. This implies 
$$\hat{t}_k-TR(\hat{p}_k,\hat{t}_k)^{-1}\leq t_k-TR(\hat{p}_k,t_k)^{-1}.$$
Since $\hat p_k$ is a $\delta$-neck center for all times $\le t_k-TR(\hat p_k,t_k)^{-1}$, it follows that $(\hat p_k,\hat t_k-TR(\hat p_k,\hat{t}_k)^{-1})$ is also the center of a $\delta$-neck in $M$. This is impossible since for large $k$, the pointed manifold $(M^n,R(\hat{p}_k,\hat{t}_k)g(-TR(\hat{p}_k,\hat{t}_k)^{-1}+\hat{t}_k),\hat{p}_k)$ is arbitrarily close to $(\operatorname{Bry}^n,g_\infty(-T),\hat{x})$ where $\hat{x}$ is the tip of the Bryant soliton. This shows that  $\hat{p}_k\in \Om_\varepsilon(t_k)$ for all large $k$. 
\end{proof}
By Lemma~\ref{lem:compactness-result}, after passing to a subsequence, the rescaled flows
\[
\left(M^n,\,R(\hat p_k,t_k)\,g\!\left(t_k+\frac{t}{R(\hat p_k,t_k)}\right),\,\hat p_k\right)_{t\le0}
\]
converge in the smooth Cheeger--Gromov sense to a $\kappa$-solution $(N^n,h(t),x_\infty)_{t\le0}$ with $R_{h(0)}(x_\infty)=1$. Since $N^n$ is a pointed Cheeger--Gromov limit of the complete, connected, noncompact manifolds $M^n$, it is also noncompact. From Lemma~\ref{lem:compactness-result}, the asymptotic shrinker of $N$ is the round cylinder $\sph^{n-1}\times \R$.
\begin{claim}
\label{claim:claim-3-noncompact-kappa-has-PIC}
 $(N^n,h(t))_{t\leq 0}$ is the canonical Ricci flow on Bryant soliton, which is uniformly $\theta_n$-PIC.
\end{claim}
\begin{proof}
Fix $s<0$, and choose $L>2|s|$. For all sufficiently large $k$, we have $t_k>\hat t_k+L\,R(\hat p_k,\hat t_k)^{-1}.$
Using Hamilton's trace Harnack inequality $R(\hat p_k,\hat t_k)^{-1}\ge R(\hat p_k,t_k)^{-1}.$ Therefore, for all sufficiently large $k$,
\[
t_k+s\,R(\hat p_k,t_k)^{-1}
>
\hat t_k+(L+s)\,R(\hat p_k,\hat t_k)^{-1}
\ge \hat t_k,
\]
since $L+s>|s|>0$. It follows that $t_k+s\,R(\hat p_k,t_k)^{-1}\in [\hat t_k,t_k)$ for all sufficiently large $k$. By the definition of $t_k$, the slice $(M,g(t_k+s\,R(\hat p_k,t_k)^{-1}))$ is uniformly $\alpha$-PIC for all sufficiently large $k$. Passing to the pointed limit gives that $(N^n,h(s))$ is uniformly $\alpha$-PIC. Since $s<0$ was arbitrary, the ancient flow $(N^n,h(t))_{t<0}$ is uniformly $\alpha$-PIC. Note that $(N^n,h(t))$ has bounded curvature and is $\kappa$-noncollapsed. Moreover, since it has nonnegative curvature operator, it is weakly PIC2. Hence $(N^n,h(t))$ is an ancient $\kappa$-solution in the sense of \cite{BN23}. By \cite{BN23}, the limit $(N^n,h(t))$ is either a shrinking cylinder (or a quotient) or the Bryant soliton. 

Since $x_\infty$ is the smooth limit of points $\hat p_k\in \Omega_\varepsilon(t_k)$, the point $x_\infty$ is not the center of an $\varepsilon/2$-neck in $(N,h(0))$. This rules out the shrinking-cylinder case, including all quotients, since in every such flow every point is the center of a neck. Hence $(N^n,h(t))$ is the Bryant soliton. In particular, by Lemma~\ref{lem:uniform-theta}, $(N^n,h(t))$ is uniformly $\theta_n$-PIC. This proves Claim~\ref{claim:claim-3-noncompact-kappa-has-PIC}.
\end{proof}
By the cap-diameter estimate \eqref{eqn:can-neighborhood-cap-diameter-estimate-l} and Claim~\ref{claim:claim-2-noncompact-kappa-has-PIC}, the entire region $\Omega_\varepsilon(t_k)$ is contained in a ball of some finite radius $r$ about $\hat p_k$, independent of $k$, with respect to the rescaled metric $R(\hat p_k,t_k)g(t_k)$. On the other hand, by Claim~\ref{claim:claim-3-noncompact-kappa-has-PIC}, after passing to a subsequence the rescaled pointed manifolds $\left(M^n,R(\hat p_k,t_k)g(t_k),\hat p_k\right)$ converge smoothly to the time-zero slice of the Bryant soliton. Hence, for all large $k$, every point of $\Omega_\varepsilon(t_k)$ satisfies $\theta_n/2$-PIC pinching inequality with respect to $g(t_k)$.

Since every point of $M\setminus\Omega_\varepsilon(t_k)$ also satisfies $\theta_n/2$-PIC pinching inequality, it follows that $(M^n,g(t_k))$ is uniformly $\theta_n/2$-PIC. By Proposition~\ref{prop:continuity-theta}, there exists $\delta>0$, depending only on $n,\theta_n,C_0$, such that $(M^n,g(t))_{t\in [t_k,t_k+\delta]}$ is uniformly $\theta_n c_n/2$-PIC. Since $\alpha<\theta_n c_n/2$, we conclude that $(M^n,g(t))_{t\in [\hat{t}_k,t_k+\delta]}$ is uniformly $\alpha$-PIC, contradicting the definition of $t_k$. This contradiction completes the proof.
\end{proof}

\begin{corollary}
\label{cor:classification-noncompact-kappa}
Let $(M^n,g(t))_{t\leq 0},n\geq 4$ be a noncompact $\kappa$-solution with asymptotic shrinker $\sph^{n-1}\times \R$. Then, $(M^n,g(t))_{t\leq 0}$ is isometric to a family of shrinking cylinders $\sph^{n-1}\times \R$, or the canonical Ricci flow on the Bryant soliton.
\end{corollary}
\begin{proof}
If the flow is the round cylinder, there is nothing to prove. Otherwise, by Lemma~\ref{lem:strong-max-principle-line-or-positive-curvature}, we may assume that $(M,g(t))_{t\leq 0}$ has positive curvature operator. From Theorem~\ref{thm:noncompact-kappa-has-PIC}, there exists $\Lambda>0$ and $b<0$ such that $(M,g(t))_{t\leq b}$ is uniformly $\Lambda$-PIC. In this case, the solution is automatically weakly PIC2. By \cite{BN23}, we conclude that $(M,g(t))_{t\leq 0}$ is isometric to the canonical Ricci flow on the Bryant soliton.
\end{proof}

\section{Compact \texorpdfstring{$\kappa$}{kappa}-solutions}
\label{sec:compact-case}
In this section we treat compact $\kappa$-solutions with asymptotic shrinker $\sph^{n-1}\times\R$. 
\begin{lemma}
\label{lem:large-diameter-times-curvature}
Let $(M^n,g(t))_{t\leq 0}$, $n\geq 4$, be a compact $\kappa$-solution with asymptotic shrinker $\sph^{n-1}\times \R$. Let $s_k\to -\infty$ be any sequence. Then, after passing to a subsequence, there exist points $q_k\in M$ such that $R(q_k,s_k)\,\operatorname{diam}_{g(s_k)}(M)^2\to \infty.$
\end{lemma}

\begin{proof}
Set $\tau_k:=-s_k\to \infty$, and let $h(\tau):=g(-\tau)$ for $\tau\ge 0$. Fix a point $p_0\in M$. By Theorem~\ref{thm:existence-of-asymptotic-shrinker}, after passing to a subsequence we may choose points $q_k\in M$ such that $\ell^h_{(p_0,0)}(q_k,\tau_k)\le \frac n2$ for all $k$, and such that the rescaled pointed flows $\left(M,\ \bar g_k(t):=\tau_k^{-1}g(\tau_k t),\ (q_k,-1)\right)_{t\in[-1,-1/2]}$ converge in the smooth pointed Cheeger--Gromov sense to the shrinking round cylinder $\sph^{n-1}\times \R$ \eqref{eqn:existence-of-asymptotic-shrinker-2}. Because $\operatorname{diam}_{\bar g_k(-1)}(M)^2=\tau_k^{-1}\operatorname{diam}_{g(s_k)}(M)^2$, it follows that $\tau_k^{-1}\operatorname{diam}_{g(s_k)}(M)^2 \to \infty.$ We also obtain that $R_{\bar g_k(-1)}(q_k)=\tau_k\,R(q_k,s_k)\to R_{g_\infty(-1)}(q_\infty)>0.$ Combining these facts, 
\[
R(q_k,s_k)\operatorname{diam}_{g(s_k)}(M)^2
=
\bigl(\tau_k R(q_k,s_k)\bigr)\,
\bigl(\tau_k^{-1}\operatorname{diam}_{g(s_k)}(M)^2\bigr)\to \infty.
\]
This completes the proof.
\end{proof}
We first use the classification in the noncompact case to identify the possible limits of the rescaled flows associated with a compact $\kappa$-solution.
\begin{lemma}
\label{lem:dim-red-compact}
Let $(M^n,g(t))_{t\leq 0},n\geq 4$ be a compact $\kappa$-solution with asymptotic shrinker $\sph^{n-1}\times \R$ and positive curvature operator. Let $t_k\to -\infty$ and $x_k\in M^n$. After passing to a subsequence, 
$$\left(M^n,g_k(t):=R(x_k,t_k)g\left({t}_k+\frac{t}{R(x_k,{t}_k)}\right),x_k\right)_{t\leq 0}\to (N^n,g_\infty(t),x_\infty)_{t\leq 0},$$
where $(N^n,g_\infty(t))$ is isometric to the family of shrinking cylinders $\sph^{n-1}\times \R$ or the Bryant soliton $\operatorname{Bry}^n$, such that $R_{g_\infty(0)}(x_\infty)=1$. 
\end{lemma}
\begin{proof}
We may use Lemma~\ref{lem:compactness-result} to obtain the existence of a $\kappa$-solution $(N^n,g_\infty(t))$ such that (i) either $N^n$ is compact and is the shrinking sphere $\sph^n$, or (ii) the asymptotic shrinker of $N^n$ is $\sph^{n-1} \times \R$. 

We first show that $N^n$ is noncompact. Suppose for contradiction that $N^n$ is compact. Then the pointed convergence is global at time $t=0$, and hence the entire rescaled manifolds $(M,g_k(0))$ converge smoothly to $(N,g_\infty(0))$. In particular, there exist constants $D,C<\infty$ such that for all sufficiently large $k$, $\operatorname{diam}_{g_k(0)}(M)\le D$ and
$\sup_M R_{g_k(0)}\le C.$ It follows that
\[
R(x_k,t_k)\,\operatorname{diam}_{g(t_k)}(M)^2=\operatorname{diam}_{g_k(0)}(M)^2\le D^2
\]
for all sufficiently large $k$. Now let $q_k\in M$ be any sequence. We obtain 
$$R(q_k,t_k)=R(x_k,t_k)\,R_{g_k(0)}(q_k)\le C\,R(x_k,t_k).$$
Therefore
\[
R(q_k,t_k)\,\operatorname{diam}_{g(t_k)}(M)^2
\le
C\,R(x_k,t_k)\,\operatorname{diam}_{g(t_k)}(M)^2
\le
CD^2
\]
for all sufficiently large $k$. This contradicts Lemma~\ref{lem:large-diameter-times-curvature}. Hence $N^n$ must be noncompact.

Therefore case (i) is impossible, and we are in case (ii). Since $N^n$ is noncompact, Corollary~\ref{cor:classification-noncompact-kappa} applies and shows that $(N^n,g_\infty(t))$ is isometric either to the family of shrinking round cylinders $\sph^{n-1}\times\R$ or to the Bryant soliton $\Bry^n$.
\end{proof}
We now complete the proof of Theorem~\ref{thm:main-result} and Theorem~\ref{thm:full-classification} in the compact case. 
\begin{theorem}
\label{classification-compact-kappa}
Let $(M^n,g(t))_{t\leq 0},n\geq 4$ be a compact $\kappa$-solution with asymptotic shrinker $\sph^{n-1}\times \R$. Then $(M^{n},g(t))_{t\leq 0}$ is isometric to the family of ancient solutions constructed by Perelman. 
\end{theorem}
\begin{proof}
Let $\pi:\widetilde M\to M$ be the universal cover, let $\Gamma$ denote the group of deck transformations, and let $\tilde g(t)=\pi^*g(t)$. By Lemma~\ref{lem:asymptotic-shrinker-of-universal-cover}, $(\widetilde M,\tilde g(t))$ is a $\kappa$-solution with asymptotic shrinker $\sph^{n-1}\times \R$. By Lemma~\ref{lem:strong-max-principle-line-or-positive-curvature}, either $(\widetilde M,\tilde g(t))$ has positive curvature operator, or it is isometric to the round cylinder $\sph^{n-1}\times \R$. 

We claim that $\widetilde M$ is compact. Suppose not. Then $(\widetilde M,\tilde g(t))_{t\le 0}$ is a noncompact $\kappa$-solution with asymptotic shrinker $\sph^{n-1}\times \R$. By Corollary~\ref{cor:classification-noncompact-kappa}, $(\widetilde M,\tilde g(t))$ is isometric either to (i) the round cylinder $\sph^{n-1}\times \R$ or to (ii) the Bryant soliton. In the first case, $(M,g(t))$ is a nontrivial quotient of $\sph^{n-1}\times \R$. But then the asymptotic shrinker of $(M,g(t))$ is the corresponding quotient cylinder, not the round cylinder $\sph^{n-1}\times \R$, contradicting the assumption. In the second case, when $(\widetilde M,\tilde g(t))$ is the Bryant soliton, every element of $\G$ fixes the tip $o\in \widetilde M$, since the tip is the unique point where the scalar curvature attains its maximum. Since $\G$ acts freely, it follows that $\G$ is trivial. In particular, $\pi$ is a diffeomorphism and $M$ is noncompact, a contradiction. Thus $\widetilde M$ must be compact. 

Therefore $(\widetilde M,\tilde g(t))$ has positive curvature operator and is compact. 
By \cite{BW08}, $\widetilde M^n$ is diffeomorphic to $\sph^n$. For each $t\leq 0$, define $\Lambda(t)$ to be the largest constant such that  $(\widetilde M,\tilde g(t))$ is uniformly $\Lambda(t)$-PIC. This is well-defined and positive because $\widetilde M^n$ is compact and $\tilde g(t)$ has positive curvature operator. 

We claim that $\inf_{t\leq 0}\Lambda(t)>0$. If not, there exists $t_k\to -\infty$ such that $\Lambda(t_k)\to 0$. This implies that there exists $p_k\in \widetilde M^n$ and an orthonormal 4-frame $\{e_1,e_2,e_3,e_4\}\sub T_{p_k}\widetilde M$ such that at $(p_k,t_k)$ one has
$$R_{1313}+R_{1414}+R_{2323}+R_{2424}-2 R_{1234} \leq 4 (2\Lambda(t_k)) |\operatorname{Rm}|.$$
Dividing both sides by $R(p_k,t_k)$, and passing to a subsequential limit using Lemma~\ref{lem:dim-red-compact}, we obtain a point $p$ on either $\Bry^n$ or the round cylinder at which $R_{1313}+R_{1414}+R_{2323}+R_{2424}-2 R_{1234}=0$. This is impossible as both the cylinder and the Bryant soliton are uniformly PIC. Thus, there exists $\Lambda>0$ such that the flow $(\widetilde M^n,\tilde g(t))_{t\leq 0}$ is uniformly $\Lambda$-PIC. Combined with positive curvature operator, which implies weakly PIC2, \cite{BDSN23} implies that $(\widetilde M^n,\tilde g(t))_{t\leq 0}$ is isometric either to a family of shrinking round spheres or to Perelman's ancient solution. Since $(\widetilde M^n,\tilde g(t))_{t\leq 0}$ has $\sph^{n-1}\times\R$ as an asymptotic shrinker, the shrinking-sphere case is impossible. Hence $(\widetilde M^n,\tilde g(t))_{t\leq 0}$ is isometric to Perelman's ancient solution.

Since Perelman's ancient solution has exactly two tips, every element of $\Gamma$ preserves the set of tips. Because $\Gamma$ acts freely, either $\Gamma$ is trivial, or $\Gamma\cong \mathbb Z_2$. In the latter case, $(M^n,g(t))_{t\leq 0}$ is the $\mathbb Z_2$-quotient of Perelman's ancient solution, and hence its asymptotic shrinker is $(\sph^{n-1}\times \R)/\mathbb Z_2$. Therefore $\Gamma$ is trivial, and hence $(M^n,g(t))_{t\leq 0}$ is isometric to Perelman's ancient solution.
\end{proof}

\end{document}